\documentclass[12pt]{amsart}

\usepackage{amssymb}

\newtheorem{theorem}{Theorem}[section]
\newtheorem{proposition}[theorem]{Proposition}
\newtheorem{lemma}[theorem]{Lemma}
\newtheorem{cor}[theorem]{Corollary}

\theoremstyle{definition}
\newtheorem{definition}[theorem]{Definition}
\newtheorem{rem}[theorem]{Remark}

\newtheorem{exam}[theorem]{Example}

\newcommand{\Gin}{\ensuremath{\mathrm{Gin}}}

\newcommand{\init}{\ensuremath{\mathrm{in}}}
\newcommand{\GL}{\ensuremath{GL_n (K)}}

\newcommand{\mideal}{\ensuremath{\mathfrak{m}}}

\newcommand{\dele}[1]{\ensuremath{\Delta^e({#1})}}
\newcommand{\dels}[1]{\ensuremath{\Delta^s({#1})}}

\newcommand{\del}[1]{\ensuremath{\Delta({#1})}}
\newcommand{\Shift}{\ensuremath{\mathrm{Shift}}}

\def\cocoa{{\hbox{\rm C\kern-.13em o\kern-.07em C\kern-.13em o\kern-.15em A}}}

\newcommand{\cyclic}{{\ensuremath{C(n,d)}}}
\newcommand{\lk}[1]{\ensuremath{\mathrm{lk}_\Gamma({#1})}}
\newcommand{\lks}[1]{\ensuremath{\mathrm{lk}_{\Shift_{ij}(\Gamma)}({#1})}}
\newcommand{\contra}[1]{\ensuremath{\mathcal{C}_\Gamma({#1})}}

\newcommand{\M}{ \ensuremath{\mathcal{M}}}

\newcommand{\sset}{\Lambda}

\begin{document}

\title
{
Algebraic shifting of
strongly edge decomposable spheres}
\author{Satoshi Murai}
\address{
Department of Mathematics\\
Graduate School of Science\\
Kyoto University\\
Sakyo-ku, Kyoto, 606-8502, Japan.
}

\maketitle

\begin{abstract}
Recently, Nevo introduced the notion of strongly edge decomposable spheres.
In this paper, we characterize algebraic shifted complexes of those spheres.
Algebraically, this result yields the characterization of the generic initial ideal
of the Stanley--Reisner ideal of Gorenstein* complexes having the strong Lefschetz property in characteristic 0.
\end{abstract}

\section{Introduction}
Algebraic shifting, which was introduced by Kalai \cite{K1,Kdisc},
is a map which associates with each simplicial complex $\Gamma$
another simplicial complex $\del \Gamma$ 
of a special type.
There are two main variants of algebraic shifting,
called \textit{exterior algebraic shifting} $\dele {\mbox{-}}$
and \textit{symmetric algebraic shifting} $\dels {\mbox{-}}$
(see \S 4 for details).
On the relation between algebraic shifting and simplicial spheres,
Kalai and Sarkaria suggested the following attractive conjecture.
Let $\del {\mbox{-}}$ be either $\dele {\mbox{-}}$ or $\dels {\mbox{-}}$
and let $\cyclic$ be the boundary complex of a cyclic $d$-polytope with $n$ vertices.
They conjectured that if $\Gamma$ is a simplicial $(d-1)$-sphere with $n$ vertices then
\begin{eqnarray} \label{0-c}
\del \Gamma \subset \Delta^s \big( \cyclic \big).
\end{eqnarray}
An important fact on this conjecture is that if a simplicial sphere $\Gamma$ satisfies
(\ref{0-c}) for either exterior algebraic shifting or symmetric algebraic shifting then the face vector of $\Gamma$ satisfies the McMullen's $g$-condition
(see \cite{Kdisc}).
Thus if the conjecture is true (for either $\dele {\mbox{-}}$ or $\dels {\mbox{-}}$) then it yields the characterization of face vectors of simplicial spheres,
which is one of the major open problems in the study of face vectors of simplicial complexes. 
However, this conjecture is widely open and only some special cases were shown (\cite{Kdisc,Mdisc,Nthesis}).
In the present paper, we will show that this conjecture is true	 for strongly edge decomposable spheres, which were introduced by Nevo \cite{N2}.

Let $\Gamma$ be a simplicial complex on $[n]=\{1,2,\dots,n\}$.
Thus $\Gamma$ is a collection of subsets of $[n]$
satisfying that (i) $\{i\} \in \Gamma$ for all $i \in [n]$
and (ii) if $F \in \Gamma$ and $G \subset F$ then $G \in \Gamma$.
An element $F$ of $\Gamma$ is called a \textit{face} of $\Gamma$
and maximal faces of $\Gamma$ under inclusion are called \textit{facets} of $\Gamma$.
A simplicial complex is said to be \textit{pure}
if all its facets have the same cardinality.
Let $f_k (\Gamma)$ be the number of faces $F \in \Gamma$ with $|F|=k+1$,
where $|F|$ is the cardinality of $F$.
The \textit{dimension} of $\Gamma$ is 
$\dim \Gamma = \max \{ k : f_k(\Gamma) \ne 0\}$.
The vector $f(\Gamma) = (f_0(\Gamma),f_1(\Gamma), \dots, f_{d-1}(\Gamma))$ is called the \textit{$f$-vector} of $\Gamma$,
where $d= \dim \Gamma +1$.
The {\em $h$-vector} $h(\Gamma)=(h_0(\Gamma),h_1(\Gamma),\dots,h_d(\Gamma))$ of
$\Gamma$ is  defined by the relations
$$h_i (\Gamma) =\sum_{j=0}^i (-1)^{i-j} { d-j \choose d-i} f_{j-1}(\Gamma)
\ \mbox{ and }\ f_{i-1}(\Gamma) = \sum_{j=0}^i { d-j \choose d-i} h_j (\Gamma),$$
where we set $f_{-1}(\Gamma) =1$.
A simplicial complex $\Gamma$ on $[n]$ is said to be \textit{shifted}
if $F \in \Gamma$ and $i \in F$ imply $(F \setminus \{i\}) \cup \{j\}
\in \Gamma$ for all $i<j \leq n $.
Note that $\dele \Gamma$ and $\dels \Gamma$ are shifted complexes
with the same $f$-vector as $\Gamma$.

First, we define strongly edge decomposable complexes.
Let $\Gamma$ be a simplicial complex on $[n]$.
The \textit{link of $\Gamma$ with respect to $F \subset [n]$}
is the simplicial complex
$$
\lk F = \{ G \subset [n] \setminus F : G \cup F \in \Gamma\}.
$$
To simplify, we write $\lk v = \lk {\{v\}}$ and $\lk {ij}=\lk {\{i,j\}}$.
Let $1 \leq i < j \leq n$ be integers.
The \textit{contraction $\contra {ij}$ of $\Gamma$ with respect to $\{i,j\}$}
is the simplicial complex on $[n] \setminus \{i\}$ which is obtained
from $\Gamma$ by identifying the vertices $i$ and $j$,
in other words,
$$\contra {ij} =
\{ F \in \Gamma: i \not \in F\} \cup \{ (F \setminus \{i\}) \cup \{j\}:
i \in F \in \Gamma\}.
$$
We say that $\Gamma$ satisfies the \textit{Link condition
with respect to $\{i,j\}$} if 
\begin{eqnarray*}
\label{1-1}
\lk i \cap \lk j  = \lk {ij}.
\end{eqnarray*}
Since $\lk i \cap \lk j$ is not empty, one has $\{i,j\} \in \Gamma$ if
$\Gamma$ satisfies the Link condition with respect to $\{i,j\}$.

\begin{definition}
\label{1no1}
The boundary complex of simplexes and $\{ \emptyset\}$ are \textit{strongly edge decomposable}
and, recursively, a pure simplicial complex $\Gamma$ is said to be
strongly edge decomposable if there exists $\{i,j\} \in \Gamma$ such that
$\Gamma$ satisfies the Link condition with respect to $\{i,j\}$
and both $\contra {ij}$ and $\lk {ij}$ are strongly edge decomposable.
\end{definition}

Definition \ref{1no1} is a natural extension of the definition of the strongly edge decomposable property
introduced by Nevo \cite[Definition 4.2]{N2}.
He assumed in addition that $\Gamma$ is a triangulated PL-manifold 
(see \cite{Hu}).
However, in our definition, strongly edge decomposable complexes
are not always PL-manifolds.
Here we give a few simple examples.

\begin{exam}
\label{exam12}
Let $\Gamma$ be the simplicial complex
generated by $\{1,2\},\{2,3\},\{3,4\}$ and $\{1,4\}$
(that is, $\Gamma$ is a cycle of length $4$).
Then $\Gamma$ satisfies the Link condition with respect to $\{1,2\}$.
Also, $\contra {\{1,2\}}$ is the boundary of the simplex $\{2,3,4\}$
and $\lk {\{1,2\}} =\{\emptyset\}$.
Hence $\Gamma$ is strongly edge decomposable.

Similarly, consider the simplicial complex $\Gamma'$
generated by $\{1,2\},\{2,3\},\{3,4\}$ and $\{2,4\}$.
Then $\Gamma'$ also satisfies the Link condition with respect to $\{1,2\}$
and we have $\mathcal{C}_{\Gamma'}(\{1,2\})=\contra {\{1,2\}}$ and
$\mathrm{lk}_{\Gamma'}(\{1,2\})=\lk {\{1,2\}}$.
Thus $\Gamma'$ is also strongly edge decomposable.
However $\Gamma'$ is not a manifold since the face $\{2\} \in \Gamma'$ is contained in $3$ different
facets.

On the other hand,
the simplicial complex $\Sigma =\Gamma \cup \{ \{1,3\}\}$
is not strongly edge decomposable since
$\Sigma$ does not satisfy the Link condition with respect to any $\{i,j\} \in \Sigma$.
\end{exam}

We will show the following.

\begin{theorem}
\label{main1}
Let $\del {\mbox{-}}$ be either $\dele {\mbox{-}}$ or $\dels {\mbox{-}}$
and let $\Gamma$ be a $(d-1)$-dimensional strongly edge decomposable complex on $[n]$.
Then $\del \Gamma$ is pure,
$h_i( \Gamma)=h_{d-i}(\Gamma)$ for $d=0,1,\dots,d$
and 
$\del \Gamma \subset \dels \cyclic$.
\end{theorem}

The relation (\ref{0-c}) is closely related to
the strong Lefschetz property of simplicial complexes (see \S 3 for the definition).
Indeed, it is known that
a $(d-1)$-dimensional simplicial complex $\Gamma$ satisfies the conditions of Theorem \ref{main1}
for symmetric algebraic shifting if and only if
$\Gamma$ has the strong Lefschetz property
in characteristic $0$.
The fact that strongly edge decomposable complexes have the strong Lefschetz property in characteristic $0$
was proved by Babson and Nevo \cite{BN}.
Thus, for symmetric algebraic shifting,
Theorem \ref{main1} follows from their result.
In this paper we prove Theorem \ref{main1}
for exterior algebraic shifting.
Unfortunately,
the method used in \cite{BN} is not applicable for exterior algebraic shifting
since the strong Lefschetz property is the condition for a quotient of a polynomial ring, however, for exterior algebraic shifting, we need to
consider exterior algebras and we do not have an analogue of
the strong Lefschetz property in exterior algebras.
To prove the result,
we use nongeneric algebraic shifting methods
developed in \cite{Mdisc}.
In particular, by applying this method,
we prove that strongly edge decomposable complexes have
the strong Lefschetz property in arbitrary characteristic.

It is known that every $2$-sphere is strongly edge decomposable,
however,
simplicial spheres are not always strongly edge decomposable
(see \cite[\S 7]{DEGN}).
In this paper,
we show that Kalai's squeezed spheres \cite{Ksq} are strongly edge decomposable
(Proposition \ref{5no4}).
This fact says that the class of strongly edge decomposable spheres is not small
since the number of combinatorial types of squeezed $(d-1)$-spheres with $n$ vertices is
larger than the number of combinatorial types of boundary complexes of simplicial $d$-polytopes with $n$ vertices
if $d \geq 5$ and $n \gg 0$.
Moreover, by using squeezed spheres, we show

\begin{theorem}
\label{main2}
Let $\del {\mbox{-}}$ be either $\dele {\mbox{-}}$ or $\dels {\mbox{-}}$
and let $\Sigma$ be a $(d-1)$-dimensional pure shifted simplicial complex
on $[n]$
satisfying $h_{i}(\Sigma)=h_{d-i}(\Sigma)$ for $i=0,1,\dots,d$
and $\Sigma \subset \dels \cyclic$.
Then there exists a squeezed sphere $\Gamma$ such that
$\del \Gamma = \Sigma$.
\end{theorem}

Theorems \ref{main1} and \ref{main2}
characterize algebraic shifted complexes of
strongly edge decomposable spheres.
Also, if $\Gamma$ is a simplicial $(d-1)$-sphere then
it is known that $\del \Gamma$ is pure and
$h_i(\Gamma)=h_{d-i}(\Gamma)$ for $i=0,1,\dots,d$.
Thus, in view of conjecture (\ref{0-c}), it is expected that the conditions in Theorem \ref{main1}
characterize algebraic shifted complexes of simplicial spheres.

To prove Theorems \ref{main1} and \ref{main2},
we first study the Link condition from an algebraic viewpoint.
It will be shown in \S 2 that the Link condition
has a nice relation to initial ideals as well as shift operators.
In particular, shift operators,
which were considered in extremal set theory,
play an important role in the study of the Link condition.
By using the above relation, we show that if a simplicial complex
$\Gamma$ satisfies the Link condition
and if its contraction and its link satisfy a nice algebraic property,
such as the Cohen--Macaulay property and the strong Lefschetz property,
then $\Gamma$ also satisfies the same property
(Propositions \ref{4no5} and \ref{3no5}).
These results and the proof of Theorems \ref{main1}
are given in \S 3 and \S 4.
Finally, in \S 5, we show that squeezed spheres are strongly edge decomposable and prove Theorem \ref{main2}.

\section{The Link condition and shift operators}

Let $\Gamma$ be a simplicial complex on $[n]$.
For given integers $1 \leq i <j \leq n$ and
for all $F \in \Gamma$,
one defines
\begin{eqnarray*}
C_{ij}(F)=
\left \{
\begin{array}{ll}
(F \setminus \{i\}) \cup \{j\}, & \mbox{ if $i \in F$, $j \not \in F$
and  $(F \setminus \{i\}) \cup \{j\} \not \in \Gamma$,}\\
F, & \mbox{ otherwise.}\end{array}
\right.
\end{eqnarray*}
Let $\Shift_{ij}(\Gamma)=\{C_{ij}(F) : F \in \Gamma\}$.
It is easy to see that $\Shift_{ij}(\Gamma)$ is a simplicial complex
and $f(\Gamma)= f(\Shift_{ij}(\Gamma))$ (see e.g., \cite[\S 8]{H}).
The operation $\Gamma \to \Shift_{ij}(\Gamma)$ was introduced by
Er\"ods--Ko--Rado \cite{EKR},
and played an important role in the classical extremal combinatorics
of finite sets (see \cite{A}).
In this section,
we study the relation between the above shift operators and the Link condition.

Let $S=K[x_1,\dots,x_n]$ be a polynomial ring over a field $K$
with each $\deg (x_i)=1$.
The \textit{Stanley--Reisner ideal $I_\Gamma$} of a simplicial complex $\Gamma$ on $[n]$
is the ideal of $S$ generated by all squarefree monomials
$x_F = \prod_{i \in F} x_i \in S$ with $F \not \in \Gamma$.
For a monomial ideal $I$ of $S$,
we write $G(I)$ for the unique minimal set of monomial generators of $I$.
For a simplicial complex $\Gamma$ on the vertex set $V$
and for a simplicial complex $\Sigma$ on the vertex set $W$
with $V \cap W = \emptyset$,
we write
$$\Gamma*\Sigma = \{ F \cup G : F \in \Gamma \mbox{ and } G \in \Sigma \}$$
and
$$ V * \Sigma = \{ F \cup G : F \subset V \mbox{ and } G \in \Sigma\}.$$
The following characterization of the Link condition is crucial for the whole paper.

\begin{lemma}
\label{2no1}
Let $\Gamma$ be a simplicial complex on $[n]$
and let $1 \leq i<j \leq n$ be integers.
The following conditions are equivalent.
\begin{itemize}
\item[(i)] $\Gamma$ satisfies the Link condition with respect to $\{i,j\}$;
\item[(ii)] $\Shift_{ij}(\Gamma)= \contra {ij} \cup \{ \{i\} \cup F : F \in \{j\} * \lk {ij} \}$;
\item[(iii)] $I_\Gamma$ has no generators which are divisible by $x_ix_j$.
\end{itemize}
In particular, if $\Gamma$ satisfies the Link condition with respect to $\{i,j\}$ then $\Shift_{ij}(\Gamma)$ also
satisfies the Link condition with respect to $\{i,j\}$.
\end{lemma}

\begin{proof}
It is clear that the contraction $\contra {ij}$ can be written in the form
\begin{eqnarray} \label{2-1}
\contra {ij} =
\{ F \in \Shift_{ij} (\Gamma) : i \not \in F \}.
\end{eqnarray}
Then
the second statement follows from (ii) since
$\Shift_{ij} ( \Shift_{ij} (\Gamma)) =\Shift_{ij} (\Gamma)$ and
$\lk {ij} = \lks {ij}$.
We will show the first statement.
\medskip

((i) $\Leftrightarrow$ (ii))
By the definition of $\Shift_{ij}$, for any $F \subset [n] \setminus \{i,j\}$, one has $\{i\} \cup F  \in \Shift_{ij}(\Gamma)$ if and only if
$\{i\} \cup F  \in \Gamma$ and $\{j\} \cup F  \in \Gamma$.
This fact says that
\begin{eqnarray}
\label{2-2}
\lks i = \big( \lk i \cap \lk j \big) \cup
\big\{ \{j \} \cup F : F \in \lk {ij} \big\}.
\end{eqnarray}
On the other hand (\ref{2-1}) says
\begin{eqnarray}
\label{2-3}
\Shift_{ij}(\Gamma) = \contra {ij} \cup
\big \{ \{ i \} \cup F : F \in \lks i \big\}.
\end{eqnarray}
Then the equivalence of (i) and (ii) follows from \medskip
(\ref{2-2}) and (\ref{2-3}).

((i) $\Rightarrow$ (iii))
Let $x_i x_j x_F \in I_\Gamma$ with $F \subset [n] \setminus \{i,j\}$.
Since $F \not \in \lk {ij} = \lk i \cap \lk j$,
we have $x_i x_F \in I_\Gamma$ or $x_j x_F \in I_\Gamma$. \medskip
Thus we have $x_ix_j x_F \not \in G(I_\Gamma)$.

((i) $\Leftarrow$ (iii))
The inclusion $\lk i \cap \lk j \supset \lk {ij}$ is obvious.
What we must prove is $\lk i \cap \lk j \subset \lk {ij}$.
Let $F \in \lk i \cap \lk j$.
Suppose $F \not\in \lk {ij}$.
Then $x_ix_jx_F \in I_\Gamma$
and there exists $x_G \in G(I_\Gamma)$ such that $G \subset \{i,j\} \cup F$.
Since $\{i\} \cup F \in \Gamma$ and $\{j \} \cup F \in \Gamma$,
we have $G \not \subset \{i\} \cup F$ and $G \not \subset \{j\} \cup F$.
Thus we have $\{i,j\} \subset G$, however,
this contradicts the assumption that $x_G \in G(I_\Gamma)$
is not divisible by $x_ix_j$.
Hence $F \in \lk {ij}$.
\end{proof}

For integers $1 \leq i < j \leq n$,
let $\varphi_{ij}$ be the graded $K$-algebra automorphism of $S$
induced 
by $\varphi_{ij}(x_k)=x_k$ for $k \ne j$ and $\varphi_{ij}(x_j)=x_i+x_j$.
We write $\init (I)$ for the initial ideal of 
a homogeneous ideal $I$ of $S$ w.r.t.\ the degree reverse lexicographic order induced by $x_1> \cdots >x_n$ (see \cite[\S 15]{E}).
Algebraically,
the benefit of Lemma \ref{2no1} (iii) can be explained by the following fact.

\begin{lemma} \label{4no3}
Let $\Gamma$ be a simplicial complex on $[n]$ and let $1 \leq i <j \leq n$ be integers.
If $I_\Gamma$ has no generators which are divisible by $x_ix_j$
then
$$\init \big(\varphi_{ij}(I_\Gamma) \big)=I_{\Shift_{ij}(\Gamma)}.$$
\end{lemma}

\begin{proof}
Since $\Gamma$ and $\Shift_{ij}(\Gamma)$ have the same $f$-vector,
$I_\Gamma$ and $I_{\Shift_{ij}(\Gamma)}$ have the same Hilbert function,
that is, $\dim_K (I_\Gamma)_k = \dim_K (I_{\Shift_{ij}(\Gamma)})_k$ for all integers $k \geq 0$,
where $I_k$ denotes the homogeneous component of degree $k$ of a homogeneous ideal $I$.
Since $I_\Gamma$ and $\init (\varphi_{ij}(I_\Gamma))$ also have the same Hilbert function,
what we must prove is  $G(I_{\Shift_{ij}(\Gamma)}) \subset \init ( \varphi_{ij}(I_\Gamma))$.
Let $x_F \in G(I_{\Shift_{ij}(\Gamma)})$.

\textit{Case 1}:
Suppose $i \not \in F$.
If $j \not \in F$ then $x_F \in I_\Gamma$.
Thus we have $\init(\varphi_{ij}(x_F)) =x_F \in \init ( \varphi_{ij}(I_\Gamma))$
as desired.
If $j \in F$ then $x_F \in I_\Gamma$ and $x_{(F \setminus \{j\}) \cup \{i\}} \in I_\Gamma$
by the definition of $\Shift_{ij}$.
Then we have
$\init ( \varphi_{ij}(x_F - x_{(F \setminus \{j\}) \cup \{i\}})) = x_F \in \init ( \varphi_{ij}(I_\Gamma))$
as desired.

\textit{Case 2}:
Suppose $i \in F$.
If $j \not \in F$ then $x_F \in I_\Gamma$ or $x_{(F \setminus \{i\}) \cup \{j \}} \in I_\Gamma$.
In both cases we have $x_F \in  \init ( \varphi_{ij}(I_\Gamma))$
since $\init(\varphi_{ij}(x_F)) =
\init(\varphi_{ij}(x_{(F \setminus \{i\}) \cup \{j \}}))
= x_F$.
If $j \in F$ then $x_F \in I_\Gamma$.
By the assumption, there exists $x_G \in G(I_\Gamma)$ such that
$G \subset F$ and $\{i,j\} \not \subset G$.
Then $\init ( \varphi_{ij}(x_G))$ is either $x_G$ or $x_{(G \setminus \{j\}) \cup \{i\}}$.
In both cases
$\init (\varphi_{ij}(x_G))$ divides $x_F$.
Hence $x_F \in \init ( \varphi_{ij}(I_\Gamma))$
as desired.
\end{proof}

\begin{rem}
Lemma \ref{4no3} is false if $I_\Gamma$ has a generator which is divisible by $x_ix_j$.
Indeed, it is easy to see that if $x_ix_jx_F \in G(I_\Gamma)$ then
$x_i^2 x_F$ is a generator of $\init(\varphi_{ij}(I_\Gamma))$.
\end{rem}

Next, we study a few simple facts on strongly edge decomposable complexes.

\begin{lemma} \label{suika}
Let $\Gamma$ be a $(d-1)$-dimensional strongly edge decomposable complex on $[n]$.
Then $\dim \contra {ij} =d-1$ and $\dim \lk {ij} =d-3$ for any $\{i,j\} \in \Gamma$. 
\end{lemma}

\begin{proof}
Since $\Gamma$ is pure,
$\dim \lk {ij} = d-3$ is obvious.
Suppose $\dim \contra {ij} <d-1$.
Then all facets of $\Gamma$ contain $\{i,j\}$.
Thus $\Gamma$ is a cone (that is, $\Gamma= \{ v\} * \lk {v}$ for some $ \{v\} \in \Gamma$).
However, if $\Gamma$ is a cone then its contraction is again a cone.
This contradicts the assumption since if $\Gamma$ is strongly edge decomposable then
we can obtain the boundary of a simplex by taking contractions repeatedly.
Thus $\dim \contra {ij} =d-1$.
\end{proof}

Let $\Gamma$ be a $(d-1)$-dimensional simplicial complex on $[n]$.
If $\Gamma$ satisfies the Link condition w.r.t.\ $\{i,j\} \in \Gamma$
then Lemma \ref{2no1} (ii) says that
\begin{eqnarray*}
f_{k}(\Gamma) = 
f_k \big( \contra {ij} \big) + f_{k-1} \big( \{j \} * \lk {ij} \big) \ \ \mbox{ for $k=0,1,\dots ,d-1$}.
\end{eqnarray*}
Moreover, if $\dim \contra {ij} = d-1$ and $\dim \lk {ij} = d-3$ then,
by using the relation between $f$-vectors and $h$-vectors, we have
\begin{eqnarray} \label{eichi}
h_k (\Gamma) &=& h_k \big(\contra {ij}\big)
+ h_{k-1} \big ( \{j \} * \lk {ij} \big ) \\
\nonumber &=& h_k \big (\contra {ij} \big) + h_{k-1} \big(\lk {ij} \big)
\end{eqnarray}
for $k=0,1,\dots,d$.
Then, arguing inductively, the $h$-vector of strongly edge decomposable complexes
satisfies the following conditions.

\begin{lemma}[Nevo] \label{2no2}
Let $\Gamma$ be a $(d-1)$-dimensional strongly edge decomposable complex.
Then $h_i(\Gamma)=h_{d-i}(\Gamma)$ for $i=0,\dots,d$ and
$h_0(\Gamma) \leq h_1(\Gamma) \leq \cdots \leq h_{\lfloor \frac d 2 \rfloor} (\Gamma)$,
where $\lfloor \frac d 2 \rfloor$ is the integer part of $\frac d 2$.
\end{lemma}

The above result was proved in \cite[Corollary 4.3]{N2}.
Actually, the $h$-vector of strongly edge decomposable complexes satisfies a stronger condition.
In the next section, we will see that strongly edge decomposable complexes have the strong Lefschetz property.
Thus the $g$-vector of those complexes is an $M$-vector 
(see \cite{S80,S}).

\section{The strong Lefschetz property}

In this section, we study the relation between the Link condition
and the strong Lefschetz property.
Let $S=K[x_1,\dots,x_n]$ be a polynomial ring over an infinite field $K$
with each $\deg x_i =1$
and $\mideal=(x_1,\dots,x_n)$ the graded maximal ideal of $S$.
For a graded $S$-module $M$,
we write $M_k$ for the homogeneous component of degree $k$ of $M$.
We refer the reader to \cite{S} for foundations on commutative algebra,
such as the Cohen--Macaulay property and linear systems of parameters.

Let $I \subset S$ be a homogeneous ideal and $A=S/I$.
Let $d$ be the Krull dimension of $A$.
We say that $A$ has the \textit{strong Lefschetz property} if $A$ is Cohen--Macaulay
and there exist a linear system of parameters (l.s.o.p.\ for short) 
$\theta_1,\dots,\theta_d \in S_1$ of $A$ and a linear form $\omega \in S_1$ such that the multiplication map
$$\omega^{s-2i} : \big(A/(\theta_1,\dots,\theta_d)A \big)_i \to \big(A/(\theta_1,\dots,\theta_d)A \big)_{s-i}$$
is bijective for $i=0,1,\dots, \lfloor \frac s 2 \rfloor$,
where $s = \max \{ k : \dim_K (A/(\theta_1,\dots,\theta_d)A)_k \ne 0\}$.
The element $\omega$ is called a \textit{strong Lefschetz element} of
$A/(\theta_1,\dots,\theta_d)A$.

Let $\Gamma$ be a $(d-1)$-dimensional simplicial complex on $[n]$.
The ring $K[\Gamma]= S /I_\Gamma$ is called the \textit{Stanley--Reisner ring of $\Gamma$}.
It is known that the Krull dimension of $K[\Gamma]$ is equal to $d$ (see \cite[II \S 1]{S}).
Let $\theta_1,\dots,\theta_d \in S_1$ be an l.s.o.p.\ of $S/I_\Gamma$.
Then it follows from \cite[II Corollary 2.5]{S} that if $\Gamma$ is Cohen--Macaulay
then
\begin{eqnarray} \label{eight}
h_i(\Gamma) = \dim_K \big( K[\Gamma]/ (\theta_1,\dots,\theta_d)K[\Gamma] \big )_i
\ \ \ \mbox{ for }i=0,1,\dots,d,
\end{eqnarray}
and $\dim_K ( K[\Gamma]/ (\theta_1,\dots,\theta_d)K[\Gamma]  )_i=0$
for $i >d$.
We say that $\Gamma$ has the strong Lefschetz property if $h_d(\Gamma) >0$
and $K[\Gamma]$ has the strong Lefschetz property.
Thus, if $\Gamma$ has the strong Lefschetz property then
$h_i(\Gamma)=h_{d-i}(\Gamma)$ for $i=0,1,\dots,d$.

We identify a sequence of linear forms $\theta_1,\dots,\theta_d \in S_1$
with an element of $K^{n \times d}$.
We require the following well-known fact
(see e.g., the proof of \cite[Theorem 4.2]{Sw}).

\begin{lemma} \label{4no2}
Let $I \subset S$ be a homogeneous ideal and $d$ the Krull dimension of $S/I$.
If $A=S/I$ has the strong Lefschetz property,
then there exits a nonempty Zariski open subset $U \subset K^{n \times (d+1)}$ such that,
for any sequence of linear forms $\theta_1,\dots,\theta_d,\theta_{d+1} \in U$,
$\theta_1,\dots,\theta_d$ is an l.s.o.p.\ of $A$ and $\theta_{d+1}$ is a strong
Lefschetz element of $A/(\theta_1,\dots,\theta_d)A$.
\end{lemma}

Let $\Gamma$ be a simplicial complex on $[n]$.
The \textit{Stellar subdivision at $F \in \Gamma$} is the operation
$\Gamma \to \mathrm{Stellar}(F,\Gamma)
=(\Gamma \setminus (F * \lk F)) \cup (\{v_F\}* \partial F * \lk F)$,
where $v_F$ is a vertex which is not contained in $[n]$ and
$\partial F$ is the boundary of the simplex generated by $F$.
It is easy to see that $\mathrm{Stellar}(F,\Gamma)$ satisfies the Link
condition w.r.t.\ $\{v_F,v\}$ and $\mathcal{C}_ {\mathrm{Stellar}(F,\Gamma)} (v_F,v)=\Gamma$ for any $v \in F$.
In \cite[Theorem 1.2]{BN}, Babson and Nevo proved that if $\Gamma$ and
$\lk F$ have the strong Lefschetz property then $\mathrm{Stellar}(F,\Gamma)$ has the strong Lefschetz property
in characteristic $0$.
In particular, in the proof of \cite[Theorem 1.2]{BN}, they essentially proved the following statement when $\mathrm{char}(K)=0$.

\begin{proposition} \label{4no5}
Let $\Gamma$ be a $(d-1)$-dimensional simplicial complex on $[n]$ satisfying the Link condition with respect to $\{i,j\}$,
where $1 \leq i < j \leq n$.
Suppose $\dim \contra {ij} = d-1$ and $\dim \lk {ij} =d-3$.
If $\contra {ij}$ and $\lk {ij}$ have the strong Lefschetz property
then $\Gamma$ has the strong Lefschetz property.
\end{proposition}

The proof given by Babson and Nevo needs the assumption
$\mathrm{char}(K)=0$ since they used the fact that if $\Gamma$ and $\Sigma$ are simplicial complexes having the strong Lefschetz property then $\Gamma*\Sigma$
also has the strong Lefschetz property when $\mathrm{char}(K)=0$
(see \cite[Theorem 2.2]{BN}). 
Here we give a more algebraic proof of Proposition \ref{4no5}
which is independent of the characteristic of the base field
by using the next fact.

\begin{lemma}[Wiebe] \label{4no4}
Let $I \subset S$ be a homogeneous ideal.
If $S/ \init(I)$ has the strong Lefschetz property
then $S/I$ has the strong Lefschetz property.
\end{lemma}

Wiebe \cite{W} proved the above statement for $\mideal$-primary homogeneous ideals.
However, one can prove it for arbitrary homogeneous ideal
in the same way as the proof of \cite[Proposition 2.9]{W}
by using Lemma \ref{4no2} and \cite[Theorem 1.1]{C}.


\begin{proof}[Proof of Proposition \ref{4no5}]
By Lemmas \ref{2no1}, \ref{4no3} and \ref{4no4},
we may assume that
$\Gamma = \Shift_{ij}(\Gamma)$.
Set $\Gamma_1 =\contra {ij}$ and $\Gamma_2 = \{j\} * \lk {ij}$.
We may assume $i=1$.
Let $S' =K[x_2,\dots,x_n]$ and let
$I_{\Gamma_1} \subset S'$ be the Stanley--Reisner ideal of $\Gamma_1$.
Let $\tilde I_{\Gamma_2}$ be the ideal of $S'$ generated by all squarefree monomials
$x_F \in S'$ with $F \not \in \Gamma_2$.
Let $A=S'/I_{\Gamma_1}$ and $B=S'/\tilde I_{\Gamma_2}$.
Then $B \cong K[\Gamma_2] \cong K[x_j] \otimes_K K[ \lk {ij}]$ as graded $K$-algebras
(see \cite[Exercise 5.1.20]{CM}), in particular,
$x_j$ is a nonzero divisor of $B$ and $B/x_jB \cong K[ \lk {ij}]$.
Hence
$B$ is Cohen--Macaulay and has the strong Lefschetz property.
Then, by Lemma \ref{4no2}, there exists a sequence of linear forms
$\theta_1,\dots,\theta_{d-1},\theta_d,\omega \in S'_1$ such that
\begin{itemize}
\item[(a)]
$\theta_1,\dots,\theta_{d-1},\theta_d$ is an l.s.o.p.\ of $A$ and $\omega$ is a strong Lefschetz element of 
$A/(\theta_1,\dots,\theta_{d-1},\theta_d)A$.
\item[(b)]
$\theta_1,\dots,\theta_{d-1}$ is an l.s.o.p.\ of $B$
and $\omega$ is a strong Lefschetz element of \linebreak 
$B/(\theta_1,\dots,\theta_{d-1})B.$
\end{itemize}
Note that 
$\max \{ k:
\dim_K (B/(\theta_1,\dots,\theta_{d-1})B)_k \ne 0\}
= d-2$ by (\ref{eight}) since $h_k (\lk {ij}) = h_k(\{j\} * \lk {ij})$ for all $k$.
\medskip

First we show that $\theta_1,\dots,\theta_{d-1},x_1-\theta_d$
is a regular sequence of $K[\Gamma]$.
Since Lemma \ref{2no1} (ii)
says $\Gamma= \Gamma_1 \cup \{ \{1\} \cup F: F \in \Gamma_2\}$,
it follows that
$K[\Gamma]$ is equal to
$$ A \oplus x_1 B \oplus x_1^2 B \oplus x_1^3 B \oplus \cdots$$ 
as $K$-vector spaces.
Then,
since $A=S'/I_{\Gamma_1}$ and $B=S'/ \tilde I_{\Gamma_2}$ are Cohen--Macaulay by the assumption,
it is clear that $\theta_1,\dots,\theta_{d-1}$
is a regular sequence of $K[\Gamma]$.
Let $R= K[\Gamma]/(\theta_1,\dots,\theta_{d-1})K[\Gamma]$,
$A'= A / (\theta_1,\dots,\theta_{d-1})A$
and $B'= B/ (\theta_1,\dots,\theta_{d-1})B$.
Then
$$R = A' \oplus x_1 B' \oplus x_1^2 B' \oplus x_1^3 B' \oplus \cdots.$$
What we must prove is that the kernel of the multiplication map
$(x_1-\theta_d) : R_s \to R_{s+1}$ is $0$ for all $s \geq 0$.
For an element $f_0 \in A'$, we write $\rho(f_0)$
for the element of $B'$ which satisfies
$x_1 \times f_0 = x_1 \rho (f_0) \in x_1 B'$ in $R$.
For any $f = f_0 + \sum_{k=1}^s x_1^k f_k \in R_s$,
where $f_0 \in A'_s$ and $f_k \in B'_{s-k}$ for $k=1,2,\dots,s$,
we have
$$(x_1-\theta_d) f = x_1^{s+1} f_s +
\left(\sum_{k=2}^s x_1^k ( f_{k-1} -\theta_d f_k) \right)
+ x_1 \big( \rho(f_0) - \theta_d f_1 \big) - \theta_d f_0.$$
Suppose $(x_1-\theta_d) f=0$.
Then $f_s =0$, $( f_{k-1} -\theta_d f_k) =0$ for $k=2,\dots,s$
and $\theta_d f_0 =0$.
Thus, inductively, we have $f_k = 0$ for $k=1,2,\dots,s$
and, since $\theta_d$ is a nonzero divisor of $A'$ by (a), we have $f_0 =0$ as desired.

Then, since $\dim \Gamma =d-1$ and since $\theta_1,\dots,\theta_{d-1},x_1-\theta_d$
is a regular sequence of $K[\Gamma]$,
it follows that $K[\Gamma]$ is Cohen--Macaulay
and the sequence $\theta_1,\dots,\theta_{d-1},x_1-\theta_d$ is an l.s.o.p.\ of $K[\Gamma]$.
\medskip

Second, we show that $\omega$ is a strong Lefschetz element of
$R/(x_1 - \theta_d)R$.
By (\ref{eichi}) and (\ref{eight}) we have
\begin{eqnarray*}
\dim_K \big( R/(x_1 - \theta_d)R \big)_k = h_k (\Gamma)= h_k \big( \contra {ij} \big) + h_{k-1} \big( \lk {ij} \big)
\ \ \mbox{ for }k=0,1,\dots,d.
\end{eqnarray*}
Then, since the assumption says that
$ h_k(\contra {ij}) = h_{d-k}(\contra {ij})$ and $h_{k-1}(\lk {ij}) = h_{d-k-1}(\lk {ij})$ for $k=0,1,\dots,d$,
we have
$$ \dim_K \big( R/(x_1 -\theta_d)R \big)_k = 
\dim \big( R/(x_1 -\theta_d)R \big)_{d-k}
\ \ \mbox{ for $k=0,1,\dots,d$}.$$ 
Thus what we must prove is that the multiplication map
$$ \omega^{d-2s} : \big( R/(x_1 -\theta_d)R \big)_s \to 
\big( R/(x_1 -\theta_d)R \big)_{d-s} $$
is surjective for $s=0,1,\dots,\lfloor  \frac d 2 \rfloor.$

Fix an integer $0 \leq s \leq \frac d 2.$ \smallskip
For any $f \in R$, write $[f]$ for its image on $R/(x_1 -\theta_d)R$.

\textit{Case 1}:
Let $x_1^k f \in x_1^k B'_{d-s-k}$ with $k \geq 1$.
We will show that there exists $g \in R_s$
such that $[\omega^{d-2s} g] = [x_1^k f]$.
If $k \geq 2$ then we have $(x_1 -\theta_d) x_1^{k-1} f \in (x_1 -\theta_d) R$,
and hence
$$[x_1^k f]=  [x_1^k f] -[x_1^k f -x_1^{k-1} \theta_d f]= [x_1^{k-1}\theta_d f].$$
Thus we may assume that $k=1$.
Then (b) says that there exists $g_0 \in B'_{s-1}$ such that
$\omega^{d-2 -2(s-1)} g_0 =f$.
Hence we have $x_1 g_0 \in R_s$ and $[\omega^{d -2s} x_1g_0 ] = [x_1 f]$
as desired.

\textit{Case 2}:
Let $f \in A'_{d-s}$.
We will show that there exists $g \in R_s$ such that
$[\omega^{d-2s}g] = [f]$.
By (a), there exists $g_0 \in A'_s$ and $h \in A'_{d-s-1}$ such that
$f-\omega^{d-2s}g_0 = \theta_d h$.
Then we have
$$[f-\omega^{d-2s}g_0] = [\theta_d h] = [\theta_d h] + [(x_1-\theta_d)h] = [x_1 \rho(h)].$$
Then, by Case 1,
there exists $g_1 \in B'_{s-1}$ such that $\omega^{d-2s}g_1=\rho(h)$.
Thus we have $g_0 + x_1 g_1 \in R_s$ and
$[\omega^{d-2s}(g_0 + x_1 g_1)] = [f]$ as desired.
\medskip

Now Case 1 and Case 2 say that the multiplication map
$\omega^{d-2s} : (R/(x_1 -\theta_d)R)_s \to (R/(x_1 -\theta_d)R)_{d-s}$
is surjective for all $s=0,1,\dots, \lfloor \frac d 2 \rfloor $.
Hence $\Gamma$ has the strong Lefschetz property.
\end{proof}

\begin{rem} \label{4no6}
The first step of the proof of Proposition \ref{4no5} says that,
with the same notation as in Proposition \ref{4no5},
if $\contra {ij}$ and $\lk {ij}$ are Cohen--Macaulay then $\Gamma$ is also Cohen--Macaulay.
\end{rem}


\begin{cor} \label{4no7}
Strongly edge decomposable complexes are Cohen--Macaulay and have the strong Lefschetz property.
\end{cor}

\begin{proof}
Let $\Gamma$ be a $(d-1)$-dimensional strongly edge decomposable complex on $[n]$.
If $\Gamma$ is the boundary of a simplex then
$\Gamma$ has the strong Lefschetz property
since $K[\Gamma]/(\theta_1,\dots,\theta_{d})K[\Gamma] \cong K[x_1]/(x_1^{d+1})$
for any l.s.o.p.\ $\theta_1,\dots,\theta_{d}$ of $K[\Gamma]$.
If $\Gamma$ is not the boundary of a simplex,
then the statement follows from Proposition \ref{4no5} and Lemma \ref{suika}
inductively.
\end{proof}

\begin{rem}
Let $\Gamma$ be a triangulated PL-sphere.
Then the link of $\Gamma$ with respect to any face $F \in \Gamma$ is again a PL-sphere.
Also, it was proved in \cite[Theorem 1.4]{N2} that
if $\Gamma$ satisfies the Link condition with respect to $\{i,j\} \in \Gamma$ then $\contra {ij}$ is also a PL-sphere.
These facts and Proposition \ref{4no5} may help to study the strong Lefschetz property of PL-spheres.
\end{rem}

\section{Algebraic shifting}

First, we recall the basics of algebraic shifting.
For further details on algebraic shifting
see the survey articles \cite{H} and \cite{K}.
Let $S=K[x_1,\dots,x_n]$ be a polynomial ring
over an infinite field $K$ with each $\deg x_i=1$
and $E= \bigwedge \langle e_1,\dots,e_n \rangle$ the exterior algebra
over $K$ with each $\deg e_i =1$.
Let $R$ be either $S$ or $E$
and let $\GL$ be the general linear group with coefficients in $K$.
Suppose that $\GL$ acts on $R$ as the group of graded $K$-algebra
automorphisms.
For a homogeneous ideal $I$ of $R$,
we write $\init (I)$ for the initial ideal of $I$ w.r.t.\
the reverse lexicographic order induced by $1>2>\cdots>n$.
The \textit{generic initial ideal} of a homogeneous ideal
$I \subset R$ is $\Gin(I)=\init (\varphi(I))$ for a generic choice of $\varphi \in \GL$ (see \cite[\S 15.9]{E} or \cite{H}).

\subsection*{Exterior algebraic shifting}
Let $\Gamma$ be a simplicial complex on $[n]$.
For a subset $F = \{ i_1,\dots,i_k\} \subset [n]$ with
$i_1 < \cdots < i_k$,
the element $e_F = e_{i_1} \wedge \cdots \wedge e_{i_k} \in E$
is called a monomial of $E$ of degree $k$.
The exterior face ideal $J_\Gamma$ of $\Gamma$
is the ideal of $E$ generated by all monomials $e_F \in E$ with $F \not \in \Gamma$.
Let $\sset$ be the set of simplicial complexes on $[n]$.
\textit{Exterior algebraic shifting} $\dele {\mbox{-}} : \sset \to \sset$
is the map defined by
$$J_{\dele \Gamma} = \Gin (J_\Gamma).$$
The simplicial complex $\dele \Gamma$ is called the
\textit{exterior algebraic shifted complex of $\Gamma$}.
Note that $\dele \Gamma$ may depend on the characteristic of the base field $K$.

\subsection*{Symmetric algebraic shifting}
Suppose $\mathrm{char}(K)=0$.
Let $\Gamma$ be a simplicial complex on $[n]$
and $I_\Gamma \subset S$ its Stanley--Reisner ideal.
Let $\M$ be the set of monomials on infinitely many variables
$x_1,x_2,\dots.$
The \textit{squarefree operation} $\Phi : \M \to \M $ is the map defined by
\begin{eqnarray*}
\Phi(x_{i_1} x_{i_2}x_{i_3} \cdots x_{i_k})
= x_{i_1} x_{i_2 +1} x_{i_3 +2} \cdots x_{i_k + k-1},
\end{eqnarray*}
where $i_1 \leq i_2 \leq  \cdots \leq i_k$.
If $I \subset S$ is a monomial ideal satisfying
$\Phi(u) \in S$ for all $u \in G(I)$,
we write $\Phi(I)$ for the monomial ideal generated by
$\{ \Phi(u): u \in G(I)\}$.
It is known that if $I \subset S$ is a squarefree monomial ideal
then $\Phi(u) \in S$ for all $u \in G(\Gin(I))$
(\cite[Lemma 1.1]{AHHshifting}).
\textit{Symmetric algebraic shifting} $\dels {\mbox{-}} : \sset \to \sset$
is the map defined by
$$J_{\dels \Gamma} = \Phi \big( \Gin(I_\Gamma) \big).$$
The simplicial complex $\dels \Gamma$ is called the
\textit{symmetric algebraic shifted complex of $\Gamma$}.
\medskip

Let $\Gamma$ and $\Sigma$ be simplicial complexes on $[n]$
and let $\del{\mbox{-}}$ be either $\dele {\mbox{-}}$ or $\dels {\mbox{-}}$.
Algebraic shifting satisfies the following properties
(see \cite{H} and \cite{K}).

\begin{itemize}
\item[($S_1$)]
$\del \Gamma$ is shifted;
\item[($S_2$)]
If $\Gamma$ is shifted then $\del \Gamma = \Gamma$;
\item[($S_3$)]
$f ( \del \Gamma )= f(\Gamma)$;
\item[($S_4$)]
If $\Sigma \subset \Gamma$ then $\del \Sigma \subset \del \Gamma$.
\end{itemize}
We need the following facts.
(The first one easily follows from \cite[Corollary 4.4]{N1},
and the second one was shown in \cite[Corollary 5.4]{N1}.)

\begin{lemma} \label{0no1}
Let $\Gamma$ be a simplicial complex on $\{m,m+1,\dots,n\}$ with $1 \leq m \leq n$
and $J_\Gamma \subset \bigwedge \langle e_m,\dots,e_n \rangle$ the exterior face ideal of $\Gamma$.
Let $J_\Gamma+(e_1,\dots,e_{m-1})$ be the ideal of $E$ generated by
$G(J_\Gamma)$ and $e_1,\dots,e_{m-1}$.
Then 
$$ \Gin \big( J_\Gamma+(e_1,\dots,e_{m-1}) \big)=\Gin(J_\Gamma)+(e_1,\dots,e_{m-1}).$$
\end{lemma}

\begin{lemma} \label{0no2}
Let $\Gamma$ be a simplicial complex on $[n]$. Then
$$\dele {\{n+1\} * \Gamma} = \{n+1\} *\dele \Gamma.$$
\end{lemma}

Note that Lemmas \ref{0no1} and \ref{0no2} are also true for symmetric algebraic shifting
(see \cite[Proposition 3.1]{CR} and \cite[Theorem 3.7]{BNT}).

In the rest of this section,
we will show that any $(d-1)$-dimensional strongly edge decomposable complex $\Gamma$ on $[n]$
satisfies $\del \Gamma \subset \dels \cyclic$.
We first recall the structure of $\dels \cyclic$.
Fix integers $n>d>0$.
For integers $i,j \in [n]$,
we write $[i] =\{1,2,\dots,i\}$
and write  $[i,j]=\{i,i+1,\dots,j-1,j\}$ if $i \leq j$ and $[i,j] = \emptyset$ if $i>j$.
A $d$-subset $F$ of $[n]$ is said to be \textit{admissible}
if $n-k \not \in F$ implies $[n-d+k,n-k-1] \subset F$.
For $i=0,1,\dots, \lfloor \frac d 2 \rfloor$, let
$$W_i(n,d)= \big\{ \big([n-d+i,n] \setminus \{n-d+i\} \big)\cup F:
F \subset [n-d+i-1],\ |F|=i \big\}$$
and
$$W_{d-i}(n,d)= \big\{ \big([n-d+i,n] \setminus \{n-i\} \big)\cup F:
F \subset [n-d+i-1],\ |F|=i \big\}.$$
Then it is easy to see that $\bigcup_{i=0}^d W_i(n,d)$ is the set of all
admissible $d$-subsets of $[n]$.
Let $\Delta(n,d)$ be the simplicial complex generated by
$\bigcup_{i=0}^d W_i(n,d)$
and let $\Delta(n,0)=\{ \emptyset \}$.
The following fact is known (see \cite[p.\ 405]{Kdisc}).

\begin{theorem}[Kalai]
Let $n>d>0$ be integers and let $\cyclic$ be the boundary complex of a cyclic $d$-polytope with $n$ vertices. Then
$$\Delta^s \big( \cyclic \big) = \Delta(n,d).$$
\end{theorem}

Note that the definition of $\Delta(n,d)$ is different from that of
\cite{Kdisc} since we reverse the ordering of the vertices.
Also, it was shown in \cite{Mdisc} that $\dele \cyclic = \dels \cyclic$.

To study Kalai and Sarkaria's conjecture,
we consider nongeneric algebraic shifting.
Let $\Gamma$ be a simplicial complex on $[n]$.
For any $\varphi \in \GL$,
we write $\Delta_\varphi(\Gamma)$ for the simplicial complex defined by
$$J_{\Delta_\varphi(\Gamma)}= \init \big( \varphi(J_\Gamma) \big).$$
The next fact can be proved in the same way as
\cite[Proposition 2.4]{Mdisc} by using \cite[Lemma 1.5]{Mdisc}.

\begin{lemma} \label{3no2}
Let $\Gamma$ be a $(d-1)$-dimensional simplicial complex on $[n]$.
If $\dele {\Delta_\varphi (\Gamma)} \subset \Delta(n,d)$ for some
$\varphi \in \GL$ then
$\dele \Gamma \subset \Delta(n,d)$.
\end{lemma}


For integers $1 \leq i<j \leq n$,
let $\varphi_{ij} \in \GL$ be the graded $K$-algebra automorphism of $E$ induced by
$\varphi_{ij}(e_k)=e_k$ for $k \ne j$ and $\varphi_{ij}(e_j)=e_i + e_j$.
We also require the following
(see \cite[Lemma 8.3]{H}).

\begin{lemma} \label{3no3}
Let $ 1 \leq i < j \leq n$ be integers.
Then, for any simplicial complex $\Gamma$ on $[n]$,
one has
$${\Shift_{ij}(\Gamma)}= \Delta_{\varphi_{ij}} (\Gamma).$$
\end{lemma}

Let $V \subset [n]$ and let $\Gamma$ be a simplicial complex on $V$.
Consider the exterior face ideal $J_\Gamma$ in $\bigwedge \langle e_k: k \in V \rangle$ and define the simplicial complex $\dele \Gamma$ on $V$ by $J_{\dele \Gamma} = \Gin (J_\Gamma)$,
and define the symmetric algebraic shifted complex $\dels \Gamma$ on $V$ similarly.
Let $C(V,d)$ be the boundary complex of a cyclic $d$-polytope
with the vertex set $V$.
Set
$\Delta(V,d)= \dels {C(V,d)}$ for $d >0$
and $\Delta(V,0)= \{ \emptyset \}$.

\begin{lemma} \label{3no4}
Let $n>d \geq 0$ be integers.
Then
$\Delta([2,n],d) \subset \Delta(n,d)$
and $\{1,n+1\}*\Delta([2,n],d) \subset \Delta(n+1,d+2)$.
\end{lemma}

\begin{proof}
If $d=0$ then the statement is obvious.
Suppose $d>0$.
Let $F$ be a facet of $\Delta([2,n],d)$.
Clearly $\Delta([2,n],d) = \{ \{ i_1,\dots,i_k\}: \{ i_1-1,\dots,i_k-1\} \in \Delta(n-1,d)\}$.
Thus there exists $0 \leq i \leq \frac d 2$ and
$F' \subset [2,n-d+i-1]$ such that
\begin{eqnarray}
\label{f1}
F= \big( [n-d+i,n] \setminus \{n-d+i\} \big) \cup F'
\end{eqnarray}
or 
\begin{eqnarray}
\label{f2}
F= \big( [n-d+i,n] \setminus \{n-i\} \big) \cup F'.
\end{eqnarray}
In both cases, it is clear that $F \in \Delta(n,d)$ and hence
$\Delta([2,n],d) \subset \Delta(n,d)$.
Also, if $F$ is an element of the form (\ref{f1}) then
$$
\{1, n+1\} \cup F= \big( [n-d+i,n+1] \setminus \{n-d+i\} \big) \cup
\big( \{ 1 \} \cup F' \big)
\in W_{i+1}(n+1,d+2)
$$
and if $F$ is an element of the form (\ref{f2}) then
$$
\{1, n+1\} \cup F= \big([n-d+i,n+1] \setminus \{n-i\} \big)
\cup \big( \{ 1 \} \cup F'  \big)
\in W_{d+2-(i+1)}(n+1,d+2).
$$
Thus we have
$\{1,n+1\}*\Delta([2,n],d) \subset \Delta(n+1,d+2)$.
\end{proof}


Now we consider the exterior algebraic shifted complex of
strongly edge decomposable complexes.
We say that a $(d-1)$-dimensional simplicial complex $\Gamma$ on $V \subset [n]$
satisfies the \textit{shifting-theoretic upper bound relation}
if $\dele \Gamma \subset \Delta(V,d)$.

\begin{proposition} \label{3no5}
Let $\Gamma$ be a $(d-1)$-dimensional simplicial complex on $[n]$ satisfying the Link condition with respect to $\{i,j\}$, where $1 \leq i<j \leq n$.
Suppose $\dim \contra {ij} = d-1$ and $\dim \lk {ij}=d-3$.
If $\contra {ij}$ and $\lk {ij}$ satisfy the shifting-theoretic
upper bound relation then $\Gamma$ satisfies the
shifting-theoretic upper bound relation. 
\end{proposition}

\begin{proof}
By Lemmas \ref{2no1}, \ref{3no2} and \ref{3no3},
we may assume that $\Gamma = \Shift_{ij}(\Gamma)$.
Let $\Gamma_1 = \contra {ij}$ and $\Gamma_2 = \{j \} * \lk {ij}$.
Since $\dele \Gamma$ is independent of the labeling of the vertices of $\Gamma$ (see \cite[p.\ 287]{BK}),
we may assume that $i=1$, $j=n$ and the vertex set of $\lk {ij}$ is
a set of the form $[m,n-1]$ for some $ m \geq 2$.

Set $E'= \bigwedge \langle e_2,\dots,e_n \rangle$
and $\tilde E = \bigwedge \langle e_m,\dots,e_n \rangle$.
Let
$J_{\Gamma_1} \subset E'$ be the exterior face ideal of $\Gamma_1$
and $J_{\Gamma_2} \subset \tilde E$ the exterior face ideal of $\Gamma_2$.
Then, since Lemma \ref{2no1} (ii) says $\Gamma= \Gamma_1 \cup \{ \{1\} \cup F : F \in \Gamma_2\}$, we have
$$J_\Gamma = J_{\Gamma_1} \bigoplus
e_1 \wedge \big( J_{\Gamma_2} +(e_2,\dots,e_{m-1}) \big),$$ 
where $J_{\Gamma_2} +(e_2,\dots,e_{m-1})$ is an ideal of $E'$.
Then there exists a $\varphi \in GL_{n-1}(K)$ which acts on $E'$
such that
$$
\init \big( \varphi(J_{\Gamma_1}) \big)=\Gin(J_{\Gamma_1})$$
and
$$
\init \big( \varphi \big( J_{\Gamma_2} + (e_2,\dots,e_{m-1}) \big) \big)
=\Gin \big( J_{\Gamma_2} +(e_2,\dots,e_{m-1}) \big).$$
Let $\tilde \varphi \in \GL$ be an automorphism of $E$
defined by $\tilde \varphi(e_1) =e_1$ and $\tilde \varphi(e_k)= \varphi (e_k)$ for $k=2,3,\dots,n$.
Then we have
$$ J_{\Delta_{\tilde \varphi}(\Gamma)} = \init \big( \tilde \varphi (J_\Gamma) \big) =
\Gin (J_{\Gamma_1}) \bigoplus e_1 \wedge \Gin \big(J_{\Gamma_2} +(e_2,\dots,e_{m-1}) \big).
$$ 
Then, since Lemmas \ref{0no1} and \ref{0no2} say that
$$\Gin \big( J_{\Gamma_2} +(e_2,\dots,e_{m-1}) \big)
= J_{\{n\} * \dele {\lk {1n}}} + (e_2,\dots,e_{m-1}),$$
we have
\begin{eqnarray*}
\Delta_{\tilde \varphi}(\Gamma) 
&=& \dele {\Gamma_1} \cup
\big\{ \{1\} \cup F: F \in \{n\} * \Delta^e \big({\lk {1n}} \big) \big\} \\
&\subset& \dele {\Gamma_1}
\cup \big( \{1,n\} * \Delta^e \big( {\lk {1n}} \big) \big).
\end{eqnarray*}
Since 
$\dele {\Gamma_1} \subset \Delta([2,n],d)$
and
$\dele {\lk {1n}} \subset \Delta([m,n-1],d-2)$
by the assumption,
Lemma \ref{3no4} says
$$\Delta_{\tilde \varphi} (\Gamma) \subset \Delta(n,d).$$
Then by ($S_2$) and ($S_4$) we have
$$\Delta^e \big( \Delta_{\tilde \varphi} (\Gamma) \big) \subset
\Delta^e \big( \Delta(n,d) \big)
= \Delta(n,d).$$
Hence we have $\dele \Gamma \subset \Delta(n,d)$ by Lemma \ref{3no2} as desired.
\end{proof}

\begin{theorem} \label{3no6}
Let $\Gamma$ be a $(d-1)$-dimensional strongly edge decomposable complex on $[n]$.
Then $\dele \Gamma \subset \Delta(n,d)$.
\end{theorem}

\begin{proof}
If $\Gamma$ is the boundary of a simplex
then we have $\dele \Gamma =\Gamma = \Delta(d+1,d)$.
If $\Gamma$ is not the boundary of a simplex,
then the statement follows from Lemma \ref{suika}
and Proposition \ref{3no5} inductively.
\end{proof}

Finally, we prove Theorem \ref{main1}.
We need the following facts
(see \cite{H} and \cite{Kdisc}).

\begin{lemma}[Kalai]
\label{CohenMacaulay}
Let $\Gamma$ be a simplicial complex.
\begin{itemize}
\item[(i)]
$\Gamma$ is Cohen--Macaulay over $K$
if and only if $\dele \Gamma$ (computed over $K$) is pure;
\item[(ii)]
$\Gamma$ is Cohen--Macaulay in characteristic $0$
if and only if $\dels \Gamma$ is pure;
\item[(iii)]
A shifted complex is Cohen--Macaulay if and only if it is pure.
\end{itemize}
\end{lemma}

\begin{lemma}[Kalai] \label{5no70}
Suppose $\mathrm{char}(K)=0$.
Let $\Gamma$ be a $(d-1)$-dimensional Cohen--Macaulay complex on $[n]$.
Then $\Gamma$ has the strong Lefschetz property if and only if
$\dels \Gamma \subset \Delta(n,d)$ and $h_i(\Gamma)=h_{d-i}(\Gamma)$ for $i=0,1,\dots,d$.
\end{lemma}

In \cite[Theorem 6.4]{Kdisc},
Kalai proved the necessity of Lemma \ref{5no70}.
However, as noted in \cite[\S 5.2]{K},
it is clear from the proof of \cite[Theorem 6.4]{Kdisc} that these conditions are equivalent.

\begin{proof}[Proof of Theorem \ref{main1}]
Let $\Gamma$ be a $(d-1)$-dimensional strongly edge decomposable complex
on $[n]$.
Since $\Gamma$ is Cohen--Macaulay in arbitrary characteristic
by Theorem \ref{4no7},
Lemma \ref{CohenMacaulay} says that $\del \Gamma$ is pure.
Also, Lemma \ref{2no2} says $h_i(\Gamma)=h_{d-i}(\Gamma)$ for $i=0,1,\dots,d$.
Finally,
$\dele \Gamma \subset \dels \cyclic$ follows from Theorem \ref{3no6}
and $\dels \Gamma \subset \dels \cyclic$ follows from
Corollary \ref{4no7} and Lemma \ref{5no70}.
\end{proof}

\section{Squeezed spheres}

Squeezed spheres were introduced by Kalai \cite{Ksq} by extending the construction of Billera--Lee polytopes \cite{BL}.
It is known that
the number of combinatorial types of squeezed $(d-1)$-spheres
with $n$ vertices is strictly larger than the number of combinatorial types of boundary complexes of simplicial $d$-polytopes with $n$ vertices
for $d \geq 5$ and $n \gg 0$ (see \cite{Ksq}).
In this section, we show that
squeezed spheres are strongly edge decomposable,
and prove Theorem \ref{main2}.

First we introduce squeezed spheres following \cite[\S 5.2]{Ksq} and \cite[\S 2]{Msq}.
Fix integers $n>d>0$ and let $m=n-d-1$.
Let $\M_{[m]}$ be the set of monomials in $K[x_1,\dots,x_m]$
where $\M_{[0]}=\{1\}$.
A set $U$ of monomials in $\M_{[m]}$ is called an
\textit{order ideal of monomials on $[m]$} if $U$ satisfies
\begin{itemize}
\item[(i)] $\{1,x_1,\dots,x_m\} \subset U$;
\item[(ii)] if $u \in U$ and $v \in \M _{[m]}$ divides $u$ then $v \in U$.
\end{itemize}
An order ideal $U$ of monomials on $[m]$ is said to be \textit{shifted}
if $u x_i \in U$ and $i<j \leq m$ imply $ux_j \in U$.
For any $u =x_{i_1} x_{i_2} \cdots x_{i_k} \in \M_{[m]}$ with $i_1 \leq i_2 \leq \cdots \leq i_k$
and with $k \leq \frac {d+1} 2$,
define a $(d+1)$-subset $F_d(u) \subset [n]$ by
$$F_d(u) = \{ i_1,i_1+1\} \cup \{ i_2+2,i_2+3\} \cup \dots \cup \{i_k +2(k-1),i_k +2k-1\}
\cup [n+2k-d,n]$$
where $F_d(1)= [n-d,n]$.

Let $U$ be a shifted order ideal of monomials of degree at most $\frac {d+1} 2$ on $[m]$
and let $B_d(U)$ be the simplicial complex generated by
$\{F_d(u): u \in U\}$.
Kalai \cite{Ksq} proved that $B_d(U)$ is a shellable $d$-ball with $n$ vertices.
This simplicial complex $B_d(U)$ is called the \textit{squeezed $d$-ball} w.r.t.\ $U$,
and its boundary $S_d(U)= \partial (B_d(U))$ is called the \textit{squeezed $(d-1)$-sphere} w.r.t.\ $U$.
The $h$-vector of the squeezed ball $B_d(U)$ and that of the squeezed sphere $S_d(U)$ are easily computed from $U$ as follows
(see \cite[Proposition 5.2]{Ksq}).

\begin{lemma}[Kalai] \label{5no1}
Let $U$ be a shifted order ideal of monomials of degree at most
$\frac {d+1} 2$ on $[m]$. Then
\begin{itemize}
\item[(i)] $h_i(B_d(U))= |\{ u \in U :\deg u =i\}|$ for $i=0,1,\dots,d+1$;
\item[(ii)] $h_i(S_d(U))-h_{i-1}(S_d(U))=h_i(B_d(U))$ for 
$i=0,1,\dots, \lfloor \frac d 2 \rfloor$.
\end{itemize}
\end{lemma}

By using the above lemma,
it is easy to see that
$f_k(B_d(U)) = f_k(S_d(U))$ for $k=0,1,\dots, \lfloor \frac d 2 \rfloor -1$
(see \cite[Proposition 5.3]{Ksq}).
Thus, in particular, we have
\begin{eqnarray} \label{5-0}
\big\{ F \in B_d(U): |F| \leq \frac d 2 \big\} =
\big\{ F \in S_d(U): |F| \leq \frac d 2 \big\}.
\end{eqnarray}

Now we will show that squeezed spheres are strongly edge decomposable.
Fix integers $d > 2$ and $n >d+1$.
Let $m=n-d-1>0$.
Let $U$ be a shifted order ideal of monomials of degree at most $\frac {d+1} 2$ on $[m]$,
and let
$$\hat U = U \cap K[x_2,\dots,x_m],$$
where $\hat U = \{1\}$ if $m =1$,
and
$$\tilde U = \{ u \in \M_{[m]}: x_1 u \in U\}.$$
Thus $U = \hat U \cup x_1 \tilde U$.
Let $B_d ( \hat U)$ be the simplicial complex generated by
$\{ F_d(u): u \in \hat U\}$.
Clearly $B_d(\hat U)$ is combinatorially isomorphic to the squeezed $d$-ball
w.r.t.\ $\{ x_{i_1} \cdots x_{i_k} \in \M_{[m-1]}: x_{i_1+1} \cdots x_{i_k+1} \in \hat U\}$.
Let $\tilde B_{d-2}(\tilde U)$ be the simplicial complex generated by $\{ F_d(x_1u) \setminus \{1,2\}: u \in \tilde U\}$.
Then $\tilde B_{d-2}( \tilde U)$ is combinatorially isomorphic to the squeezed $(d-2)$-ball w.r.t.\
$\{ x_{i_1} \cdots x_{i_k} \in \M_{[m-\ell]}: x_{i_1+\ell} \cdots x_{i_k+\ell} \in \tilde U\}$,
where $\ell= \max \{ k : 0 \leq k \leq m,\ x_k \not \in \tilde U\}$.
We
write $S_d(\hat U)$ for the boundary of $B_d(\hat U)$
and write $\tilde S_{d-2} (\tilde U)$ for the boundary of $\tilde B_{d-2} (\tilde U)$.

\begin{lemma} \label{5no2}
With the same notations as above,
one has
\begin{eqnarray} \label{5-1}
\Shift_{12} \big( S_d(U) \big) =
S_d(\hat U) \cup \big\{ \{1 \} \cup F: F \in \{2 \} * \tilde S_{d-2}( \tilde U) \big\}.
\end{eqnarray}
In particular,
$\mathcal{C}_{S_d(U)} ( \{1,2\}) = S_d(\hat U)$,
$\mathrm{lk}_{S_d(U)}(\{1,2\})= \tilde S_{d-2}(\tilde U)$
and $S_d(U)$ satisfies the Link condition with respect to $\{1,2\}$.
\end{lemma}

\begin{proof}
By using (\ref{2-1}) together with the equation
$\lk {ij} = \lks {ij}$,
the second statement immediately follows from (\ref{5-1})
and Lemma \ref{2no1} (ii).
Thus what we must prove is equation (\ref{5-1}).
First, we show that the lefthand side contains the righthand side.
\smallskip

\textit{Case 1}:
Let $F$ be a facet of $S_d(\hat U)$.
We will show $F \in \Shift_{12}(S_d(U))$.
Since $1 \not \in F$,
if $F$ is a facet of $S_d(U)$ then we have
$C_{12}(F) =F \in \Shift_{12}(S_d(U))$ as desired.

Suppose that $F$ is not a facet of $S_d(U)$.
Then since $F$ is a facet of $S_d( \hat U)$, there exists the unique monomial $u_0 \in \hat U$
such that $F \subset F_d(u_0) \in B_d(\hat U)$.
Since $F$ is not a facet of $S_d(U)$, there exists $x_1 v \in x_1\tilde U$ such that
$F \subset F_d(x_1 v)$.
In particular, since $\{1,2\} \subset F_d(x_1v)$ and $1 \not \in F$,
we have $F= F_d(x_1v) \setminus \{1\}$.

We will show $F_d(x_1v) \setminus \{2\} \in S_d(U)$.
By the definition of $F_d(\mbox{-})$
it is clear that, for any $u \in U$,
if $1 \in F_d(u)$ then $2 \in F_d(u)$.
Thus if $G$ is a facet of $B_d(U)$ which contains $F_d(x_1v) \setminus \{2\}$ then $G$ must be equal to $F_d(x_1v)$.
Thus $F_d(x_1v)$ is the only facet of $B_d(U)$ which contains $F_d(x_1v) \setminus \{2\}$.
Hence $F_d(x_1v) \setminus \{2\}$ is a facet of $S_d(U)$.
Then we have 
$C_{12}(F_d(x_1v) \setminus \{2\}) = F_d(x_1v) \setminus \{ 1\} = F \in \Shift_{12}(S_d(U))$ as desired.
\smallskip

\textit{Case 2}:
Let $F$ be a facet of $\tilde S_{d-2} (\tilde U)$.
We will show $\{1,2\} \cup F \in \Shift_{12}(S_d(U))$.
Since $F$ is a facet of $\tilde S_{d-2}(\tilde U)$,
there exits the unique monomial $u_0 \in \tilde U$ such that
$\{1,2\} \cup F \subset F_d(x_1u_0)$.
However,
since $U= \hat U \cup x_1 \tilde U$ and $1 \not \in F_d(u)$ for any $u \in \hat U$,
$F_d(x_1u_0)$ is the only facet of $B_d(U)$ which contains $\{1,2\} \cup F$.
Thus we have
$\{1,2\} \cup F \in S_d(U)$
and $C_{12}(\{1,2\} \cup F) = \{1,2\} \cup F \in \Shift_{12}(S_d(U))$ as desired.
\smallskip

We already proved that
$\Shift_{12}(S_d(U)) \supset S_d(\hat U) \cup \{ \{1\} \cup F:
F \in \{2\} * \tilde S_{d-2} (\tilde U) \}$.
Thus, to prove (\ref{5-1}), it is enough to show that
\begin{eqnarray} \label{XX}
f_k(S_d(U)) = f_k(S_d(\hat U)) +
f_{k-1}(\{2\} * \tilde S_{d-2}(\tilde U))
\ \ \mbox{ for all $k$}.
\end{eqnarray}
By Lemma \ref{5no1}, we have
\begin{eqnarray*}
&&h_k \big( S_d(\hat U) \big) + h_{k-1} \big( \{2\} * \tilde S_{d-2}(\tilde U) \big)\\
&&= h_k \big( S_d(\hat U) \big) + h_{k-1} \big( \tilde S_{d-2}(\tilde U) \big)\\
&&= |\{ u \in \hat U: \deg u \leq k\}| + |\{ u \in \tilde U: \deg u \leq k-1\}|\\
&&= |\{ u \in U: \deg u \leq k\}|\\
&&= h_k \big( S_d(U) \big)
\end{eqnarray*}
for $k=0,1,\dots, \lfloor \frac d 2 \rfloor$.
Then, by the Dehn-Sommerville equations, we have 
$$h_k \big( S_d(\hat U) \big) + h_{k-1} \big( \{2\} * \tilde S_{d-2}(\tilde U) \big)
=h_k \big( S_d(\hat U) \big) + h_{k-1} \big( \tilde S_{d-2}(\tilde U) \big)
= h_k \big( S_d(U) \big) $$ for all $k$.
By using the above equations as well as the relation between $f$-vectors and $h$-vectors,
a routine computation implies the desired equation (\ref{XX}).
\end{proof}

\begin{rem} \label{5no3}
The set of facets of squeezed spheres was completely determined in \cite[Proposition 1]{L}.
It would yield an alternate proof of Lemma \ref{5no2}.
\end{rem}

\begin{proposition} \label{5no4}
Squeezed spheres are strongly edge decomposable.
\end{proposition}

\begin{proof}
Let $S_d(U)$ be a squeezed $(d-1)$-sphere with $n$ vertices.
We use induction on $d$ and $n$.
If $d=1$ or $n =d+1$ then the statement is obvious
since $S_d(U)$ is the boundary of a simplex.
Also, it is easy to see that any $1$-dimensional sphere is strongly edge decomposable.
Indeed, if $\Gamma$ is a $1$-dimensional sphere with $k$ edges, where $k \geq 4$,
then, for any $\{i,j\} \in \Gamma$,
$\Shift_{ij}(\Gamma)= \contra {ij} \cup \{\{i,j\}, \{i\}\}$ 
and $\contra{ij}$ is a $1$-dimensional sphere with $(k-1)$ edges.
Finally,
if $d > 2$ and $n>d+1$ then the statement follows from Lemma \ref{5no2} and the induction hypothesis.
\end{proof}

Next we will prove Theorem \ref{main2}.
For a $(d-1)$-dimensional simplicial complex $\Gamma$ on $[n]$,
let
$$U(\Gamma)=\{ u \in \M _{[n-d-1]}: u \not \in \Gin(I_\Gamma)\}.$$
To prove Theorem \ref{main2}, we need the following facts
which were shown in \cite[Proposition 4.1, Theorem 4.2 and Corollary 7.7]{Msq}.
Let $\Phi : \M \to \M$ be the squarefree operation defined in \S 4.

\begin{lemma} \label{5no9}
Let $n>d>0$ be integers and $m=n-d-1$.
Let $U$ be a shifted order ideal of monomials of degree at most
$\frac {d+1} 2$ on $[m]$
and let $I(U) \subset S$ be the ideal generated by
$\{ u \in \M_{[m]}: u \not \in U\}$.
\begin{itemize}
\item[(i)]
$I_{\dele {B_d(U)}}= I_{\dels {B_d(U)}} = \Phi(I(U))$;
\item[(ii)] If $\max \{ \deg u : u \in U \} \leq \frac d 2$ then
$U(S_d(U))=U.$
\end{itemize}
\end{lemma}

It was noted in \cite[\S 5.1]{K} that if one has $\Delta ( S_d(U)) \subset \Delta(n,d)$ for all squeezed spheres $S_d(U)$ on $[n]$
then Lemma \ref{5no9}
yields Theorem \ref{main2} without a proof.
To prove this,
we require the following fact
(see \cite[p.\ 398]{Kdisc} or \cite[Lemma 3.4]{Msq}).

\begin{lemma}[Kalai] \label{5noA}
Assume $\mathrm{char}(K)=0$.
Let $\Gamma$ and $\Sigma$ be $(d-1)$-dimensional simplicial complexes on $[n]$ having the strong Lefschetz property.
Then
\begin{itemize}
\item[(i)] $U(\Gamma)$ is a shifted order ideal of monomials of degree at most $\frac d 2$ on $[n-d-1]$;
\item[(ii)] if $U(\Gamma) = U(\Sigma)$ then
$\dels \Gamma = \dels \Sigma$.
\end{itemize}
\end{lemma}

\begin{proposition} \label{5noB}
If $S_d(U)$ is a squeezed $(d-1)$-sphere then $\dele {S_d(U)} = \dels {S_d(U)}$.
\end{proposition}

\begin{proof}
By $(S_2)$ (defined in \S 4) it is enough to show $\dels {\dele {S_d(U)}} = \dels { \dels {S_d(U)}}$.
Lemma \ref{CohenMacaulay} says that
$\dele {S_d(U)}$ and $\dels {S_d(U)}$ are Cohen--Macaulay.
Since squeezed spheres are strongly edge decomposable, by using $(S_2)$,
Theorem \ref{main1} and Lemma \ref{5no70} say that
$\dele {S_d(U)}$ and $\dels {S_d(U)}$ have the strong Lefschetz property in characteristic $0$.
Then, by Lemma \ref{5noA}, what we must prove is
\begin{eqnarray} \label{7-0}
&& \big\{ u \in \M _{[n-d-1]}:u \not \in  \Gin ( I_{\dele {S_d(U)}} ) \big\}
=\big\{ u \in \M _{[n-d-1]}:u \not \in  \Gin(I_{ \dels {S_d(U)}}) \big\}.
\end{eqnarray}
Here $\Gin(\mbox{-})$ is the generic initial ideal in characteristic $0$.
Lemma \ref{5noA} also says that
sets of monomials which appear in (\ref{7-0})
are sets of monomials of degree at most $\frac d 2$.
Then (\ref{5-0}) and Lemma \ref{5no9} (i) say
\begin{eqnarray*}
&& \big\{ u \in \M _{[n-d-1]}:u \not \in  \Gin(I_{\dele {S_d(U)}})\big\}\\
&&=\left\{ u \in \M _{[n-d-1]}:u \not \in  \Gin(I_{\dele {B_d(U)}}),\
\deg u \leq  \frac d 2\right\}\\
&&=\left\{ u \in \M _{[n-d-1]}:u \not \in  \Gin(I_{\dels {B_d(U)}}),\
\deg u \leq  \frac d 2 \right\}\\
&&= \big\{ u \in \M _{[n-d-1]}:u \not \in  \Gin(I_{\dels {S_d(U)}}) \big\},
\end{eqnarray*}
as desired.
\end{proof}

Now we will prove Theorem \ref{main2}.

\begin{proof}[Proof of Theorem \ref{main2}]
By Proposition \ref{5noB},
it is enough to show the statement for symmetric algebraic shifting.
Thus we may assume $\mathrm{char}(K)=0$.
Since $\Sigma$ is shifted and pure,
$\Sigma$ is Cohen--Macaulay
by Lemma \ref{CohenMacaulay}.
Also, since $\dels \Sigma = \Sigma$ by $(S_2)$,
the assumption and Lemma \ref{5no70} say that $\Sigma$ has
the strong Lefschetz property.
Then Lemma \ref{5noA} says that
$U(\Sigma)$
is a shifted order ideal of monomials of degree at most
$\frac d 2$. 
Then Lemma \ref{5no9} (ii) says
$U(S_d( U(\Sigma)))=U(\Sigma)$.
Since $S_d(U(\Sigma))$ has the strong Lefschetz property by Corollary \ref{4no7} and Proposition \ref{5no4}, we have $\dels {S_d(U(\Sigma))} = \dels \Sigma =\Sigma$ by Lemma \ref{5noA}.
\end{proof}

\begin{rem}
In general, exterior algebraic shifting depends on the characteristic of
the base field,
however, Proposition \ref{5noB} says that
the exterior algebraic shifted complex of squeezed spheres
is independent of the characteristic of the base field.
Also, it is possible to compute the facets of
$\dele {S_d(U)} = \dels {S_d(U)}$ from $U$
by using \cite[Theorem 6.4]{Kdisc} and Lemma \ref{5no9} as follows.

For a $(d-1)$-dimensional Cohen--Macaulay complex $\Gamma$ on $[n]$,
let
$$L(\Gamma)= \big\{ u \in \M _{[n-d]}: u \not \in \Gin(I_\Gamma) \big\},$$
where $\Gin(I_\Gamma)$ is the generic initial ideal of $I_\Gamma$ in characteristic $0$.
For a homogeneous ideal $I$ of $S$,
we write $I_{\leq k}$ for the ideal of $S$ generated by all polynomials
in $I$ of degree at most $k$.
It follows from \cite[(6.3)]{Kdisc} that the set of facets of
$\dels {\Gamma}$ is
\begin{eqnarray} \label{7-A}
&&\bigcup_{k=0}^d \big\{ \{i_1,i_2+1,\dots,i_k +k-1\} \cup [n-d+1+k,n]:
x_{i_1}\cdots x_{i_k} \in L(\Gamma) \big\},
\end{eqnarray}
where $i_1 \leq \cdots \leq i_k$.
Moreover, if $\Gamma$ has the strong Lefschetz property then
$U(\Gamma)$ determines $L(\Gamma)$ by the relation
\begin{eqnarray} \label{7-B}
L(\Gamma)= \{ u x_{n-d}^t: u \in U(\Gamma), \ 0 \leq t \leq d-2 \deg u\}.
\end{eqnarray}
(See \cite[p.\ 398]{Kdisc} or \cite[Lemma 3.4]{Msq}.)
On the other hand, by using Lemma \ref{5no9},
it is not hard to show that
$U(S_d(U)) = \{ u \in U: \deg u \leq \frac d 2\}$.
Indeed, if we set $U'= \{ u \in U: \deg u \leq \frac d 2\}$,
then Lemma \ref{5no9} (i) and (\ref{5-0}) say that
$(I_{\dels {S_d(U)}})_{\leq \frac d 2} =(I_{\dels {S_d(U')}})_{\leq \frac d 2}$.
This fact says that $\Gin(I_{S_d(U)})_{\leq \frac d 2} = \Gin (I_{S_d(U')})_{\leq \frac d 2}$,
and hence $U(S_d(U))=U(S_d(U'))=U'$ by Lemma \ref{5no9} (ii).
Then the facets of $\dels {S_d(U)}$ are determined from (\ref{7-A}) and (\ref{7-B}).
\end{rem}

\begin{rem}
Corollary \ref{4no7} and Proposition \ref{5no4} give an affirmative answer to \cite[Problem 4.5]{Msq}.
This yields the characterization of the generic initial ideal of the Stanley--Reisner ideal
of Gorenstein* complexes (see \cite[p.\ 67]{S}) having the strong Lefschetz property in characteristic $0$.
Indeed Theorem \ref{main2} and Lemmas \ref{CohenMacaulay} and \ref{5no70}
characterize the symmetric algebraic shifted complex of those complexes,
and knowing $\dels \Gamma$ is equivalent to knowing $\Gin(I_\Gamma)$ in characteristic $0$.
Also,
by using the relation between generic initial ideals and generic hyperplane sections (\cite[Corollary 2.15]{G}),
this characterization can be extended to the characterization of generic initial ideals of homogeneous ideals $I$ 
which satisfy that
$S/I$ is a Gorenstein homogeneous $K$-algebra having the strong Lefschetz property in characteristic $0$.
\end{rem}

\bigskip

\noindent
\textbf{Acknowledgements.}
I would like to thank Eran Nevo for helpful comments on an earlier version of this paper.
The author is supported by JSPS Research Fellowships for Young Scientists

\end{document}